\documentclass[USenglish,onecolumn]{article}
\usepackage[big]{dgruyter_NEW}
\usepackage{amsthm,amsmath,amssymb,amsfonts}

\usepackage{mathtools}
\mathtoolsset{showonlyrefs=true} % Only numbers equations that are referenced
\usepackage{dsfont} % for the \ind \newcommand
\usepackage{mathrsfs} % for the \mathscr command

\newtheorem{corollary}{Corollary}

\newtheorem{remark}{Remark}
\newtheorem{theorem}{Theorem}
\newtheorem{lemma}{Lemma}

\newcommand{\N}{\mathbb{N}}
\newcommand{\R}{\mathbb{R}}

\newcommand{\EE}{\mathbb{E}}
\newcommand{\BB}{\mathbb{B}\mathrm{ias}}
\newcommand{\VV}{\mathbb{V}\mathrm{ar}}

\newcommand{\bb}[1]{\boldsymbol{#1}}

\newcommand{\OO}{\mathcal O}
\newcommand{\oo}{\mathrm{o}}
\newcommand{\leqdef}{\vcentcolon=}

\newcommand{\rd}{{\rm d}}

\newcommand{\ind}{\mathds{1}}

\allowdisplaybreaks

\begin{document}

\articletype{Research Article{\hfill}Open Access}

\author[1]{Fr\'ed\'eric Ouimet}

\affil[1]{Centre de recherches math\'ematiques, Universit\'e de Montr\'eal, Montreal, QC H3T 1J4, Canada. Department of Mathematics and Statistics, McGill University, Montreal, QC H3A 0B9, Canada. Division of Physics, Mathematics and Astronomy, California Institute of Technology, Pasadena, CA 91125, USA.}

\title{\huge{On the boundary properties of Bernstein estimators on the simplex}}

\runningtitle{On the boundary properties of Bernstein estimators on the simplex}

%\journalname{Open Statistics}

%\startpage{01}

%\journalyear{2022}
%\journalvolume{3}
%\journalissue{1}

\renewcommand{\lq}{\textquotedblleft}

%\DOI{10.1515/demo-2022-xxxx}
%\received{***, 2022}
%\revised{***, 2022}
%\accepted{October 25, 2022}

\maketitle

\vspace{-3mm}
{\small
\noindent
\textbf{Abstract:}
In this paper, we study the asymptotic properties (bias, variance, mean squared error) of Bernstein estimators for cumulative distribution functions and density functions near and on the boundary of the $d$-dimensional simplex. Our results generalize those found by \citet{MR2925964}, who treated the case $d=1$, and complement the results from \citet{MR4287788} in the interior of the simplex. Since the ``edges'' of the $d$-dimensional simplex have dimensions going from $0$ (vertices) up to $d - 1$ (facets) and our kernel function is multinomial, the asymptotic expressions for the bias, variance and mean squared error are not straightforward extensions of one-dimensional asymptotics as they would be for product-type estimators studied by almost all past authors in the context of Bernstein estimators or asymmetric kernel estimators. This point makes the mathematical analysis much more interesting.}



\section{Introduction}\label{sec:intro}

The $d$-dimensional (unit) simplex and its interior are defined by
\begin{equation}\label{eq:def.simplex}
    \mathcal{S}_d \leqdef \big\{\bb{x}\in [0,1]^d: \|\bb{x}\|_1 \leq 1\big\} \quad \text{and} \quad \mathrm{Int}(\mathcal{S}_d) \leqdef \big\{\bb{x}\in (0,1)^d: \|\bb{x}\|_1 < 1\big\},
\end{equation}
where $\|\bb{x}\|_1 \leqdef \sum_{i=1}^d |x_i|$.
For any cumulative distribution function $F$ on $\mathcal{S}_d$ (meaning that it takes the values $0$ or $1$ outside $\mathcal{S}_d$), define the Bernstein polynomial of order $m$ for $F$ by
\begin{equation}\label{eq:Bernstein.polynomial}
    F_m^{\star}(\bb{x}) \leqdef \sum_{\bb{k}\in \N_0^d \cap m \mathcal{S}_d} F(\bb{k}/m) P_{\bb{k},m}(\bb{x}), \quad \bb{x}\in \mathcal{S}_d, ~m \in \N,
\end{equation}
where the weights are the following probabilities from the $\text{Multinomial}\hspace{0.2mm}(m,\bb{x})$ distribution:
\begin{equation}\label{eq:multinomial.probability}
    P_{\bb{k},m}(\bb{x}) \leqdef \frac{m!}{(m - \|\bb{k}\|_1)! \prod_{i=1}^d k_i!} (1 - \|\bb{x}\|_1)^{m - \|k\|_1} \prod_{i=1}^d x_i^{k_i}, \quad \bb{k}\in \N_0^d \cap m\mathcal{S}_d.
\end{equation}
(Here, the set $m \mathcal{S}_d$ has the obvious meaning $\{\bb{y}\in \R^d : \bb{y} = m \bb{x} ~\text{for some } \bb{x}\in \mathcal{S}_d\}$.)
The Bernstein estimator of $F$, denoted by $F_{n,m}^{\star}$, is the Bernstein polynomial of order $m$ for the empirical cumulative distribution function
\begin{equation}
    F_n(\bb{x}) \leqdef n^{-1} \sum_{i=1}^n \ind_{(-\bb{\infty},\bb{x}]}(\bb{X}_i),
\end{equation}
where the observations $\bb{X}_1, \bb{X}_2, \dots, \bb{X}_n$ are assumed to be independent and $F$ distributed.
Precisely, for $m,n \in \N$, let
\begin{equation}\label{eq:cdf.Bernstein.estimator}
    F_{n,m}^{\star}(\bb{x}) \leqdef \sum_{\bb{k}\in \N_0^d \cap m \mathcal{S}_d} F_n(\bb{k}/m) P_{\bb{k},m}(\bb{x}), \quad \bb{x}\in \mathcal{S}_d.
\end{equation}
It should be noted that the c.d.f.\ estimator in \eqref{eq:cdf.Bernstein.estimator} only makes sense here if the observations' support is contained in an hyperrectangle inside the unit simplex.
If the observations have full support on the unit simplex, then the c.d.f.\ estimator would take values in $(0,1)$ on the interior of the unit hypercube.

\vspace{2mm}
For a density $f$ supported on $\mathcal{S}_d$, we define the Bernstein density estimator of $f$ by
\begin{equation}\label{eq:histogram.estimator}
    \widetilde{f}_{n,m}(\bb{x}) \leqdef \hspace{-2mm}\sum_{\bb{k}\in \N_0^d \cap (m-1) \mathcal{S}_d} \hspace{-2mm} m^d \left\{\frac{1}{n} \sum_{i=1}^n \ind_{(\frac{\bb{k}}{m}, \frac{\bb{k} + 1}{m}]}(\bb{X}_i)\right\} P_{\bb{k},m-1}(\bb{x}), \quad \bb{x}\in \mathcal{S}_d,
\end{equation}
where $m^d$ is a scaling factor equal to the inverse of the volume of the hypercube $\big(\tfrac{\bb{k}}{m}, \tfrac{\bb{k} + 1}{m}\big] \leqdef \big(\frac{k_1}{m}, \frac{k_1 + 1}{m}\big] \times \dots \times \big(\frac{k_d}{m}, \frac{k_d + 1}{m}\big]$.
%The scaling factor here is slightly different than the one for $\hat{f}_{n,m}$ in \cite{MR4287788}.
As pointed out in \cite{MR4287788}, the estimator $\widetilde{f}_{n,m}$ is not a proper density but integrates to $1$ asymptotically, and it can be written as a finite mixture of Dirichlet densities (in the one-dimensional case, finite beta mixtures are studied for example in \cite{MR1789474}).

\section{Review of the literature and motivation}\label{sec:brief.review}

    \cite{MR0397977} was the first author to consider Bernstein density estimation on the compact interval $[0,1]$, namely \eqref{eq:histogram.estimator} with $d = 1$.
    The author computes the asymptotics of the bias, variance and mean squared error (MSE) at each point where the second derivative of $f$ exists (also assuming that $f$ is bounded everywhere). His proof rests on careful Taylor expansions for the density terms inside the bulk of the binomial distribution while concentration bounds are applied to show that the contributions coming from outside the bulk are negligible.
    The optimal rate, with respect to the MSE, is achieved when $m \asymp n^{2/5}$ and shown to be $\OO_x(n^{-4/5})$ for $x\in (0,1)$.
    For $x\in \{0,1\}$ (at the boundary), he finds that the MSE is $\OO_x(n^{-3/5})$ using $m \asymp n^{2/5}$, but this choice of $m$ turns out to be suboptimal; the rate $\OO_x(n^{-2/3})$ can be achieved with $m \asymp n^{1/3}$.
    These results were generalized by \cite{MR0574548,MR0638651} to target densities with support of the form $[0,1]$, $[0,\infty)$ and $(-\infty,\infty)$, where the binomial weights of the density estimator were replaced by more general convolutions.

    \vspace{2mm}
    \cite{MR1293514} was the first author to consider Bernstein estimation in the multivariate context.
    Assuming that $f$ is twice continuously differentiable, he derived asymptotic expressions for the bias, variance and MSE of the density estimators on the two-dimensional unit simplex and the unit square $[0,1]^2$, and also proved their uniform strong consistency and asymptotic normality. He showed that the optimal rate of the MSE is $\OO_{\bb{x}}(n^{-2/3})$ for interior points when $d = 2$, and it is achieved when $m \asymp n^{1/3}$.
    On the ``edges'' of dimension $1$ (i.e., excluding the corners), he showed that the optimal rate is $\OO_{\bb{x}}(n^{-4/7})$ and it is achieved when $m \asymp n^{2/7}$.
    On the ``edges'' of dimension $0$ (i.e., the three corners of the $2$-dimensional simplex), he showed that the optimal rate is $\OO_{\bb{x}}(n^{-1/2})$ and it is achieved when $m \asymp n^{1/4}$. As we will see, this deterioration of the MSE near the boundary is solely due to an increase of the variance of the density estimator, while the bias remains asymptotically uniform and of order $\OO(m^{-1})$.
    This is in contrast with traditional kernel estimators, where the opposite trade-off happens: the bias increases near the boundary while the variance remains asymptotically the same.
    For both classes of estimators, the rates of the MSE might be the same on the boundary but for different reasons.

    \vspace{2mm}
    It should be noted that this variance increase near and on the boundary was previously observed by other authors for product-type asymmetric kernel estimators (meaning when the kernel is a product of one-dimensional kernels) on the geometrically simpler domains $[0,1]^d$ and $[0,\infty)^d$, see, e.g., \cite[page~142~and~(7a)]{MR2568128}, \cite[pages~50~and~52]{MR3979322} and \cite[Corollary~1]{MR4322320}.
    \cite{MR2925964} also observed this phenomenon in one dimension and provided explicit asymptotic expressions for the bias, variance and MSE of the classical Bernstein density and c.d.f.\ estimators near and on the boundary of $[0,1]$. Our main goal in this paper is to extend Leblanc's results to the $d$-dimensional simplex domain and complete the theoretical analysis initiated by \cite{MR4287788} on the asymptotic properties of Bernstein density and c.d.f.\ estimators on the simplex (with multinomial weights).

    \vspace{2mm}
    In \cite{MR2925964}, Leblanc extended the results of \cite{MR0397977} by not only studying the asymptotics of the bias and variance for points that are on the boundary but also {\it near} the boundary (i.e., for $x = \lambda / m$ where $\lambda \geq 0$ remains fixed as $n,m\to \infty$). In particular, he showed that the leading term of the variance can be written using modified Bessel functions of the first kind. He also proved similar results for the Bernstein c.d.f.\ estimator. As mentioned above, our goal in this paper is to generalize the results of \cite{MR2925964} to all $d\geq 1$. The results will help us understand how the dimension of the ``edge'' we are closest to influences the variance and the MSE of the density estimator and the c.d.f.\ estimator.

    \vspace{2mm}
    Since the ``edges'' of the $d$-dimensional simplex have dimensions going from $0$ (vertices) up to $d - 1$ (facets) and our kernel function is multinomial, the asymptotic expressions for the bias, variance and MSE are not straightforward extensions of the one-dimensional expressions as they would be for product-type estimators such as the ones in \cite[page~142~and~(7a)]{MR2568128}, \cite[pages~50~and~52]{MR3979322} and \cite[Corollary~1]{MR4322320}. This is why the results in the present paper are of mathematical interest. Apart from the present paper, the only article dealing with boundary results for multidimensional non product-type Bernstein estimators or asymmetric kernel estimators that we are aware of seems to be Corollary~2 of \cite{MR4322320}, where the variance of a multivariate elliptical-based Birnbaum–Saunders kernel density estimator is computed near and on the boundary.

    \vspace{2mm}
    Other asymptotic results (not directly related to the boundary, which is our focus here) have been proved in the unidimensional case by \cite{MR1910059,MR2068610,MR2179543,MR2351744,MR2488150,MR2662607,MR2960952,MR3174309,MR3412755,MR4130895,MR4335173,MR4374864}, among other authors, and in the multidimensional case by \cite{MR2270097,MR3474765,MR4287788}. Most of these papers (and more) are reviewed in Section~2 of \cite{MR4287788}.
    Similarly, for a detailed overview of the literature on the closely related Beta and Dirichlet kernel estimators, we refer to Section 2 of \cite{MR4319409}.

\section{Contribution, outline and notation}

\subsection{Contribution}

    Our theoretical contribution is to find asymptotic expressions for the bias, variance and MSE of the Bernstein c.d.f.\ and density estimators, defined respectively in \eqref{eq:cdf.Bernstein.estimator} and \eqref{eq:histogram.estimator}, {\it near} and on the boundary of the $d$-dimensional simplex.
    We also deduce the asymptotically optimal choice of the bandwidth parameter $m$ using the expressions for the MSE, which can be used in practice to implement a plug-in selection method.
    All these results generalize those found in \cite{MR2925964} for the unit interval and complement those found in \cite{MR4287788} in the interior of the simplex.
    Our rates of convergence for the MSE are in line with those recently found in \cite{MR4319409} for Dirichlet kernel estimator.
    In both cases, the general rule is that the variable smoothing that is built-in for Bernstein estimators and asymmetric kernel estimators yields an asymptotically smaller bias near the boundary at the cost of an increase in variance, compared to traditional multivariate kernel estimators.

    \vspace{2mm}
    Under the assumption that the target density is twice continuously differentiable, we find in particular that the variance is $\OO_{\bb{x}}(n^{-1} m^{d/2})$ in the interior of the simplex and it gets multiplied by a factor $m^{1/2}$ everytime we go near the boundary in one of the $d$ dimensions. If we are near an edge of dimension $d - |\mathcal{J}|$ (see Section~\ref{sec:results.density.estimator} for the definition of $\mathcal{J}$), then the variance is $\OO_{\bb{x}}(n^{-1} m^{(d + |\mathcal{J}|)/2})$. Additional smoothness conditions on the partial derivatives of the target density can improve those rates, see, e.g., Corollary~\ref{cor:Leblanc.2012.boundary.Corollary.1.next} for more details on this point.

    \vspace{2mm}
    In contrast to other methods of boundary bias reduction (such as the reflection method or boundary kernels (see, e.g., \cite[Chapter 6]{MR3329609}), this property is built-in for Bernstein estimators, which makes them one of the easiest to use in the class of estimators that are asymptotically unbiased near (and on) the boundary. Bernstein estimators are also non-negative everywhere on their domain, which is definitely not the case of many estimators corrected for boundary bias.
    This is another reason for their desirability.
    Also, as an anonymous referee pointed out, the fact that the bandwidth parameter $m$ is an integer makes the optimisation step easier to implement (for most bandwidth selection criteria like least-square cross validation, likelihood cross-validation, etc.) than for estimators where the bandwidth parameter $h$ is a real number.
    Bandwidth selection methods and their consistency will be investigated thoroughly in upcoming work.

    \subsection{Outline}

    In Section~\ref{sec:results.density.estimator} and Section~\ref{sec:results.cdf.estimator}, we state our results for the density estimator and the c.d.f.\ estimator, respectively.
    The proofs are given in Section~\ref{sec:proofs.results.density.estimator} and Section~\ref{sec:proofs.results.cdf.estimator}.
    Some technical lemmas and tools are gathered in Appendix~\ref{sec:tech.lemmas}.

    \subsection{Notation}

    Throughout the paper, the notation $u = \OO(v)$ means that $\limsup |u / v| < C < \infty$ as $m$ or $n$ tend to infinity, depending on the context.
    The positive constant $C$ can depend on the target c.d.f.\ $F$, the target density $f$ or the dimension $d$, but no other variable unless explicitly written as a subscript. One common occurrence is a local dependence of the asymptotics with a given point $\bb{x}$ on the simplex, in which case we would write $u = \OO_{\bb{x}}(v)$.
    In a similar fashion, the notation $u = \oo(v)$ means that $\lim |u / v| = 0$ as $m$ or $n$ tend to infinity.
    The same rule applies for the subscript.
    The symbol $\mathscr{D}$ over an arrow `$\longrightarrow$' will denote the convergence in law (or distribution).
    We will use the shorthand $[d] \leqdef \{1,2,\dots,d\}$ in several places.
    The functions
    \begin{equation}\label{eq:def:I.0.I.1}
        I_{\nu}(z) \leqdef \sum_{k=0}^{\infty} \frac{(z/2)^{2k + \nu}}{k! (k + \nu)!}, \quad \nu\in \{0,1\},
    \end{equation}
    will denote the modified Bessel functions of the first kind of order $0$ and $1$, respectively.
    For any vector $\bb{v}\in \R^d$ and any subset of indices $A\subseteq [d]$, we write
    \begin{equation}\label{eq:vector.subset}
        \bb{v}_A \leqdef (v_i)_{i\in A} \quad \text{and} \quad \psi_A(\bb{x}) \leqdef \left[(4\pi)^{|A|} \, \bigg(1 - \sum_{i\in A} x_i\bigg) \prod_{i\in A} x_i\right]^{-1/2},
    \end{equation}
    with the conventions $\sum_{\emptyset} \leqdef 0$ and $\prod_{\emptyset} \leqdef 1$.
    Finally, the bandwidth parameter $m = m(n)$ is always implicitly a function of the number of observations, the only exceptions being in Lemmas~\ref{lem:Leblanc.2012.boundary.Lemma.2},~\ref{lem:Leblanc.2012.boundary.Lemma.6},~\ref{lem:Leblanc.2012.boundary.Lemma.3},~\ref{lem:Leblanc.2012.boundary.Lemma.7} and the related proofs.

\section{Results for the density estimator \texorpdfstring{$\widetilde{f}_{n,m}$}{hat(f)n,m}}\label{sec:results.density.estimator}

For every result stated in this section, we will make the following assumption:

\begin{equation}\label{eq:assump:f.density}
    \begin{aligned}
        \bullet \quad
        &\text{$f\in C^2(\mathrm{Int}(\mathcal{S}_d))$ and there exists an open set $\mathcal{U}\subseteq \R^d$ that contains $\mathcal{S}_d$ and } \\[-1mm]
        &\text{an extension $f_{\text{ext}} : \mathcal{U}\to \R$ such that $f_{\text{ext}} \equiv f$ on $\mathcal{S}_d$ and $f_{\text{ext}}\in C^2(\mathcal{U})$.}
    \end{aligned}
\end{equation}

\begin{remark}\label{rem:f.C.2.S.meaning}
    If $f\in C^2(\R^d)$, then we can just take $f_{\text{ext}} = f$.
    However, if $f$ is discontinuous anywhere at some point on the boundary of $\mathcal{S}_d$, then, at that point, the partial derivatives of $f$ will technically refer to the partial derivatives of $f_{\text{ext}}$.
\end{remark}

In the first lemma, we obtain a general expression for the bias of the density estimator.

\begin{lemma}[Bias of $\widetilde{f}_{n,m}(\bb{x})$ on $\mathcal{S}_d$]\label{lem:Leblanc.2012.boundary.Lemma.2}
    Assume \eqref{eq:assump:f.density}.
    Then, uniformly for $\bb{x}\in \mathcal{S}_d$, we have, as $m\to \infty$,
    \begin{equation}\label{eq:lem:Leblanc.2012.boundary.Lemma.2}
        \BB[\widetilde{f}_{n,m}(\bb{x})] = m^{-1} \Delta_1(\bb{x}) + m^{-2} \Delta_2(\bb{x}) + \frac{1}{m^2} \sum_{i\in [d]} \oo\left(1 + \sqrt{(m-1) x_i (1 - x_i)} + (m - 1) x_i (1 - x_i)\right),
    \end{equation}
    \vspace{-1mm}
    where
    \begin{equation}\label{eq:def:b.x}
        \begin{aligned}
            \Delta_1(\bb{x})
            &\leqdef \sum_{i\in [d]} \big(\tfrac{1}{2} - x_i\big) \, \frac{\partial}{\partial x_i} f(\bb{x}) + \frac{1}{2} \sum_{i,j\in [d]} \big(x_i \ind_{\{i = j\}} - x_i x_j\big) \, \frac{\partial^2}{\partial x_i \partial x_j} f(\bb{x}), \\[-1mm]
            \Delta_2(\bb{x})
            &\leqdef \sum_{i,j\in [d]} \big(\tfrac{1}{6} \ind_{\{i = j\}} + \tfrac{1}{8} \ind_{\{i \neq j\}} - \tfrac{1}{2} x_i \ind_{\{i = j\}} - \tfrac{1}{2} x_j + x_i x_j\big) \, \frac{\partial^2}{\partial x_i \partial x_j} f(\bb{x}).
        \end{aligned}
    \end{equation}
\end{lemma}

By considering points $\bb{x}\in \mathcal{S}_d$ that are close to the boundary in some components (see the subset of indices $\mathcal{J}\subseteq [d]$ below), we get the bias of the density estimator near the boundary.

\begin{theorem}[Bias of $\widetilde{f}_{n,m}(\bb{x})$ near the boundary of $\mathcal{S}_d$]\label{thm:Leblanc.2012.boundary.Theorem.1}
    Assume \eqref{eq:assump:f.density}.
    For any $\bb{x}\in \mathcal{S}_d$ such that $x_i = \lambda_i / m$ for all $i\in \mathcal{J}\subseteq [d]$ ($\lambda_i\geq 0$ is fixed $\forall i\in \mathcal{J}$) and $x_i\in (0,1)$ is independent of $m$ for all $i\in [d] \backslash \mathcal{J}$, we have, as $n\to \infty$,
    \begin{align}\label{eq:thm:Leblanc.2012.boundary.Theorem.1}
        \BB(\widetilde{f}_{n,m}(\bb{x}))
        &= \frac{b_{\mathcal{J}}(\bb{x})}{m} + \frac{1}{m^2} \left\{\hspace{-1mm}
            \begin{array}{l}
                -\sum_{i\in [d]} \lambda_i \, \left.\frac{\partial}{\partial x_i} f(\bb{x})\right|_{\bb{x}_{\mathcal{J}} = \bb{0}} \\
                + \sum_{i,j\in [d]} \big((\tfrac{1}{6} + \lambda_i) \ind_{\{i = j\}} + (\tfrac{1}{8} + \tfrac{\lambda_j}{2}) \ind_{\{i \neq j\}}\big) \left.\frac{\partial^2}{\partial x_i \partial x_j} f(\bb{x})\right|_{\bb{x}_{\mathcal{J}} = \bb{0}}
            \end{array}
            \hspace{-1mm}\right\} \notag \\
        &\quad+ \oo_{\bb{\lambda}_{\mathcal{J}}}\big(m^{-2} + \ind_{\{\mathcal{J}\neq [d]\}} m^{-1}\big),
    \end{align}
    where $\bb{x}_{\mathcal{J}}$ and $\bb{\lambda}_{\mathcal{J}}$ are defined in \eqref{eq:vector.subset}, and where
    \begin{equation}\label{eq:def.b.J.x}
        b_{\mathcal{J}}(\bb{x}) \leqdef \sum_{i\in [d]} \big(\tfrac{1}{2} - x_i \ind_{\{i\in [d]\backslash \mathcal{J}\}}\big) \left.\frac{\partial}{\partial x_i} f(\bb{x})\right|_{\bb{x}_{\mathcal{J}} = \bb{0}} + \sum_{i,j\in [d]\backslash \mathcal{J}} \tfrac{1}{2} \big(x_i \ind_{\{i = j\}} - x_i x_j\big) \left.\frac{\partial^2}{\partial x_i \partial x_j} f(\bb{x})\right|_{\bb{x}_{\mathcal{J}} = \bb{0}}.
    \end{equation}
\end{theorem}

Next, we obtain a general expression for the variance of the density estimator.

\begin{lemma}[Variance of $\widetilde{f}_{n,m}(\bb{x})$ on $\mathcal{S}_d$]\label{lem:Leblanc.2012.boundary.Lemma.4}
    Assume \eqref{eq:assump:f.density}.
    For any $\bb{x}\in \mathcal{S}_d$ such that $x_i = \lambda_i / m$ for all $i\in \mathcal{J}\subseteq [d]$ ($\lambda_i\geq 0$ is fixed $\forall i\in \mathcal{J}$) and $x_i\in (0,1)$ is independent of $m$ for all $i\in [d] \backslash \mathcal{J}$, we have, as $n\to \infty$,
    \begin{equation}\label{eq:lem:Leblanc.2012.boundary.Lemma.4}
        \begin{aligned}
            \VV(\widetilde{f}_{n,m}(\bb{x}))
            &= \frac{m^d}{n} f(\bb{x}) \hspace{-3mm} \sum_{\bb{k}\in \N_0^d \cap (m-1)\mathcal{S}_d} \hspace{-3mm} P_{\bb{k},m-1}^2(\bb{x}) \\
            &\quad+ \frac{m^{d-1}}{n} \sum_{i\in [d]} \OO\left(1 + \sqrt{\sum_{\bb{k}\in \N_0^d \cap (m-1)\mathcal{S}_d} \hspace{-3mm} |k_i - m x_i|^2 P_{\bb{k},m-1}(\bb{x})} \sqrt{\sum_{\bb{k}\in \N_0^d \cap (m-1)\mathcal{S}_d} P_{\bb{k},m-1}^3(\bb{x})}\right).
        \end{aligned}
    \end{equation}
\end{lemma}

By combining Lemma~\ref{lem:Leblanc.2012.boundary.Lemma.4} and the technical estimate in Lemma~\ref{lem:Leblanc.2012.boundary.Lemma.3}, we get the asymptotics of the variance of the density estimator near the boundary.

\begin{theorem}[Variance of $\widetilde{f}_{n,m}(\bb{x})$ near the boundary of $\mathcal{S}_d$]\label{thm:Leblanc.2012.boundary.Theorem.2}
    Assume \eqref{eq:assump:f.density}.
    For any $\bb{x}\in \mathcal{S}_d$ such that $x_i = \lambda_i / m$ for all $i\in \mathcal{J}\subseteq [d]$ ($\lambda_i\geq 0$ is fixed $\forall i\in \mathcal{J}$) and $x_i\in (0,1)$ is independent of $m$ for all $i\in [d] \backslash \mathcal{J}$, we have, as $n\to \infty$,
    \begin{equation}\label{eq:thm:Leblanc.2012.boundary.Theorem.2}
        \VV(\widetilde{f}_{n,m}(\bb{x})) = n^{-1} m^{(d + |\mathcal{J}|)/2} \left\{v_{\mathcal{J}}(\bb{x}) + \OO_{\bb{\lambda}_{\mathcal{J}}}(m^{-1}) + \oo_{\bb{x}_{[d]\backslash \mathcal{J}}}(1) \, \ind_{\{\mathcal{J} \neq [d]\}}\right\},
    \end{equation}
    where $\bb{x}_{[d]\backslash \mathcal{J}}$ and $\bb{\lambda}_{\mathcal{J}}$ are defined in \eqref{eq:vector.subset}, and where
    \begin{equation}\label{eq:def.v.J.x}
        v_{\mathcal{J}}(\bb{x}) \leqdef \left.f(\bb{x})\right|_{\bb{x}_{\mathcal{J}} = \bb{0}} \, \psi_{[d]\backslash \mathcal{J}}(\bb{x}) \prod_{i\in \mathcal{J}} e^{-2 \lambda_i} I_0(2 \lambda_i).
    \end{equation}
\end{theorem}

By combining Theorem~\ref{thm:Leblanc.2012.boundary.Theorem.1} and Theorem~\ref{thm:Leblanc.2012.boundary.Theorem.2}, we get the asymptotics of the mean squared error of the density estimator near and on the boundary.
In particular, the optimal bandwidth parameter $m$ will depend on the number of components of $\bb{x}$ that are close to the boundary.

\begin{corollary}[MSE of $\widetilde{f}_{n,m}(\bb{x})$ near the boundary of $\mathcal{S}_d$]\label{cor:Leblanc.2012.boundary.Corollary.1}
    Assume \eqref{eq:assump:f.density}.
    For any $\bb{x}\in \mathcal{S}_d$ such that $x_i = \lambda_i / m$ for all $i\in \mathcal{J}\subseteq [d]$ ($\lambda_i\geq 0$ is fixed $\forall i\in \mathcal{J}$) and $x_i\in (0,1)$ is independent of $m$ for all $i\in [d] \backslash \mathcal{J}$, we have, as $n\to \infty$,
    \begin{equation}\label{eq:cor:Leblanc.2012.boundary.Corollary.1}
        \begin{aligned}
            \mathrm{MSE}(\widetilde{f}_{n,m}(\bb{x}))
            &\leqdef \EE\big[(\widetilde{f}_{n,m}(\bb{x}) - f(\bb{x}))^2\big] = \VV(\widetilde{f}_{n,m}(\bb{x})) + \big(\BB[\widetilde{f}_{n,m}(\bb{x})]\big)^2 \\[1.5mm]
            &= n^{-1} m^{(d + |\mathcal{J}|)/2} v_{\mathcal{J}}(\bb{x}) + m^{-2} b_{\mathcal{J}}^2(\bb{x}) + n^{-1} m^{(d + |\mathcal{J}|)/2} \, \big[\OO_{\bb{\lambda}_{\mathcal{J}}}(m^{-1}) \\
            &\quad+ \oo_{\bb{x}_{[d]\backslash \mathcal{J}}}(1) \, \ind_{\{\mathcal{J} \neq [d]\}}\big] + \OO_{\bb{\lambda}_{\mathcal{J}}}(m^{-3}) + \oo_{\bb{\lambda}_{\mathcal{J}}}\big(m^{-2}\big) \, \ind_{\{\mathcal{J}\neq [d]\}},
        \end{aligned}
    \end{equation}
    where $\bb{x}_{\mathcal{J}}$, $\bb{x}_{[d]\backslash \mathcal{J}}$ and $\bb{\lambda}_{\mathcal{J}}$ are defined in \eqref{eq:vector.subset}, $b_{\mathcal{J}}(\bb{x})$ is defined in \eqref{eq:def.b.J.x}, and $v_{\mathcal{J}}(\bb{x})$ is defined in \eqref{eq:def.v.J.x}.
    If the quantity inside the big bracket is non-zero in \eqref{eq:cor:Leblanc.2012.boundary.Corollary.1}, the asymptotically optimal choice of $m$, with respect to MSE, is
    \begin{equation}
        m_{\mathrm{opt}} = n^{\frac{2}{d + |\mathcal{J}| + 4}} \frac{\left\{b_{\mathcal{J}}(\bb{x})\right\}^{\frac{4}{d + |\mathcal{J}| + 4}}}{\left\{\frac{d + |\mathcal{J}|}{4} v_{\mathcal{J}}(\bb{x})\right\}^{\frac{2}{d + |\mathcal{J}| + 4}}},
    \end{equation}
    ($\psi_{[d]\backslash \mathcal{J}}$ is defined in \eqref{eq:vector.subset}) in which case, we have, as $n\to \infty$,
    \begin{equation}
        \begin{aligned}
            \mathrm{MSE}(\widetilde{f}_{n,m_{\mathrm{opt}}}(\bb{x}))
            &= n^{-\frac{4}{d + |\mathcal{J}| + 4}} \frac{\left\{b_{\mathcal{J}}(\bb{x})\right\}^{\frac{2(d + |\mathcal{J}|)}{d + |\mathcal{J}| + 4}}}{\left[\frac{\big(\frac{4}{d + |\mathcal{J}|}\big)^{\frac{4}{d + |\mathcal{J}| + 4}}}{\frac{4}{d + |\mathcal{J}|} + 1}\right] \left\{v_{\mathcal{J}}(\bb{x})\right\}^{\frac{-4}{d + |\mathcal{J}| + 4}}} \\
            &\quad+ \OO_{\bb{\lambda}_{\mathcal{J}}\hspace{-0.2mm},\hspace{0.3mm} \bb{x}_{[d]\backslash \mathcal{J}}}\big(n^{-\frac{6}{d + |\mathcal{J}| + 4}} + \ind_{\{\mathcal{J} \neq [d]\}} n^{-\frac{5}{d + |\mathcal{J}| + 4}}\big) + \oo_{\bb{\lambda}_{\mathcal{J}}}\big(n^{-\frac{4}{d + |\mathcal{J}| + 4}}\big) \, \ind_{\{\mathcal{J}\neq [d]\}}.
        \end{aligned}
    \end{equation}
\end{corollary}

By imposing further conditions on the partial derivatives of $f$, we can remove terms from the bias in Theorem~\ref{thm:Leblanc.2012.boundary.Theorem.1} and obtain another expression for the mean squared error of the density estimator near and on the boundary, and also the corresponding optimal bandwidth parameter $m$ when $\mathcal{J} = [d]$.

\begin{corollary}[MSE of $\widetilde{f}_{n,m}(\bb{x})$ near the boundary of $\mathcal{S}_d$]\label{cor:Leblanc.2012.boundary.Corollary.1.next}
    Assume \eqref{eq:assump:f.density} and also
    \begin{equation}\label{cor:Leblanc.2012.boundary.Corollary.1.next.condition}
        \left.\frac{\partial}{\partial x_i} f(\bb{x})\right|_{\bb{x}_{\mathcal{J}} = \bb{0}} = 0 \quad \forall i\in [d], \quad (\text{\it shoulder condition}) \qquad \left.\frac{\partial^2}{\partial x_i \partial x_j} f(\bb{x})\right|_{\bb{x}_{\mathcal{J}} = \bb{0}} = 0 \quad \forall (i,j)\in ([d]\backslash \mathcal{J})^2,
    \end{equation}
    (in particular, the first bracket in \eqref{eq:thm:Leblanc.2012.boundary.Theorem.1} is zero).
    Then, for any $\bb{x}\in \mathcal{S}_d$ such that $x_i = \lambda_i / m$ for all $i\in \mathcal{J}\subseteq [d]$ ($\lambda_i\geq 0$ is fixed $\forall i\in \mathcal{J}$) and $x_i\in (0,1)$ independent of $m$ for all $i\in [d] \backslash \mathcal{J}$, we have, as $n\to \infty$,
    \begin{equation}\label{eq:cor:Leblanc.2012.boundary.Corollary.1.next}
        \begin{aligned}
            \mathrm{MSE}(\widetilde{f}_{n,m}(\bb{x}))
            &= n^{-1} m^{(d + |\mathcal{J}|)/2} v_{\mathcal{J}}(\bb{x}) + m^{-4} L_{\mathcal{J}}^2(\bb{x}) \\[0.5mm]
            &\quad+ n^{-1} m^{(d + |\mathcal{J}|)/2} \, \big[\OO_{\bb{\lambda}_{\mathcal{J}}}(m^{-1}) + \oo_{\bb{x}_{[d]\backslash \mathcal{J}}}(1) \, \ind_{\{\mathcal{J} \neq [d]\}}\big] \\
            &\quad+ \oo_{\bb{\lambda}_{\mathcal{J}}\hspace{-0.2mm},\hspace{0.3mm} \bb{x}_{[d]\backslash \mathcal{J}}}\big(m^{-4} + \ind_{\{\mathcal{J}\neq [d]\}} m^{-3}\big),
        \end{aligned}
    \end{equation}
    where $\bb{x}_{\mathcal{J}}$, $\bb{x}_{[d]\backslash \mathcal{J}}$ and $\bb{\lambda}_{\mathcal{J}}$ are defined in \eqref{eq:vector.subset}, $v_{\mathcal{J}}(\bb{x})$ is defined in \eqref{eq:def.v.J.x}, and where
    \begin{equation}
        L_{\mathcal{J}}(\bb{x}) \leqdef \sum_{(i,j)\in [d]^2\backslash ([d]\backslash \mathcal{J})^2} \big((\tfrac{1}{6} + \lambda_i) \ind_{\{i = j\}} + (\tfrac{1}{8} + \tfrac{\lambda_j}{2}) \ind_{\{i \neq j\}}\big) \left.\frac{\partial^2}{\partial x_i \partial x_j} f(\bb{x})\right|_{\bb{x}_{\mathcal{J}} = \bb{0}}.
    \end{equation}
    Note that the last error term in \eqref{eq:cor:Leblanc.2012.boundary.Corollary.1.next} is bigger than the main term except when $\mathcal{J} = [d]$.
    Therefore, if $\mathcal{J} = [d]$ and we assume that the quantity inside the big bracket is non-zero in \eqref{eq:cor:Leblanc.2012.boundary.Corollary.1.next}, the asymptotically optimal choice of $m$, with respect to MSE, is
    \begin{equation}
        m_{\mathrm{opt}} = n^{\frac{1}{d + 4}} \frac{\left\{L_{[d]}(\bb{0})\right\}^{\frac{2}{d + 4}}}{\left\{\frac{d}{4} v_{[d]}(\bb{0})\right\}^{\frac{1}{d + 4}}},
    \end{equation}
    so that
    \begin{equation}
        \mathrm{MSE}(\widetilde{f}_{n,m_{\mathrm{opt}}}(\bb{x})) = n^{-\frac{4}{d + 4}} \frac{\left\{L_{[d]}(\bb{0})\right\}^{\frac{2d}{d + 4}}}{\left[\frac{\big(\frac{4}{d}\big)^{\frac{4}{d + 4}}}{\frac{4}{d} + 1}\right] \left\{v_{[d]}(\bb{0})\right\}^{\frac{-4}{d + 4}}} + \oo_{\bb{\lambda}}\big(n^{-\frac{4}{d + 4}}\big), \quad \text{as } n\to \infty.
    \end{equation}
\end{corollary}

\begin{remark}
    The last part shows that, when we assume the smoothness conditions \eqref{cor:Leblanc.2012.boundary.Corollary.1.next.condition}, the optimal rate of the MSE near any vertex of the simplex is of the same order as in the interior of simplex (Corollary~\ref{cor:Leblanc.2012.boundary.Corollary.1} with $\mathcal{J} = \emptyset$) without the smoothness conditions.
\end{remark}

\begin{remark}
    In order to optimize $m$ when $\mathcal{J} \neq [d]$ in Corollary~\ref{cor:Leblanc.2012.boundary.Corollary.1.next}, we would need a more precise expression for the bias in Theorem~\ref{eq:thm:Leblanc.2012.boundary.Theorem.1} by assuming more regularity conditions on $f$ than we did in \eqref{eq:assump:f.density}.
    %We have not tried to do so because the number of terms to manage in the proof of Theorem~\ref{eq:thm:Leblanc.2012.boundary.Theorem.1} is already hard to track.
\end{remark}

\section{Results for the c.d.f.\ estimator \texorpdfstring{$F_{n,m}^{\star}$}{Fn,m star}}\label{sec:results.cdf.estimator}

For every result stated in this section, we will make the following assumption:
\begin{equation}\label{eq:assump:F.cdf}
    \bullet \quad \text{$F$ is three-times continuously differentiable on $\mathcal{S}_d$.}
\end{equation}
Below, we obtain a general expression for the bias of the c.d.f.\ estimator on the simplex, and then near the boundary.

\begin{lemma}[Bias of $F_{n,m}^{\star}(\bb{x})$ on $\mathcal{S}_d$]\label{lem:Leblanc.2012.boundary.Lemma.6}
    Assume \eqref{eq:assump:F.cdf}.
    Then, uniformly for $\bb{x}\in \mathcal{S}_d$, we have, as $m\to \infty$,
    \begin{equation}\label{eq:lem:Leblanc.2012.boundary.Lemma.6}
        \begin{aligned}
            &\EE[F_{n,m}^{\star}(\bb{x})] \\
            &\quad= F(\bb{x}) + \frac{1}{2m} \sum_{i,j\in [d]} (x_i \ind_{\{i = j\}} - x_i x_j) \, \frac{\partial^2}{\partial x_i \partial x_j} F(\bb{x}) \\
            &\qquad+ \frac{1}{6m^2} \sum_{i,j,\ell\in [d]} \left(2 x_i x_j x_{\ell} - \ind_{\{i = j\}} x_i x_{\ell} - \ind_{\{j = \ell\}} x_i x_j - \ind_{\{i = \ell\}} x_j x_{\ell} + \ind_{\{i = j = \ell\}} x_i\right) \frac{\partial^3}{\partial x_i \partial x_j \partial x_{\ell}} F(\bb{x}) \\
            &\qquad+ \frac{1}{m^3} \sum_{i,j,\ell\in [d]} \oo\left(\big(\EE\big[|\xi_i - m \, x_i|^2\big]\big)^{1/2} \big(\EE\big[|\xi_j - m \, x_j|^4\big]\big)^{1/4} \big(\EE\big[|\xi_{\ell} - m \, x_{\ell}|^4\big]\big)^{1/4}\right),
        \end{aligned}
    \end{equation}
    where $\bb{\xi} = (\xi_1,\xi_2,\dots,\xi_d)\sim \mathrm{Multinomial}\hspace{0.2mm}(m,\bb{x})$.
\end{lemma}

\newpage
\begin{theorem}[Bias of $F_{n,m}^{\star}(\bb{x})$ near the boundary of $\mathcal{S}_d$]\label{thm:Leblanc.2012.boundary.Theorem.3}
    Assume \eqref{eq:assump:F.cdf}.
    For any $\bb{x}\in \mathcal{S}_d$ such that $x_i = \lambda_i / m$ for all $i\in \mathcal{J}\subseteq [d]$ ($\lambda_i > 0$ is fixed $\forall i\in \mathcal{J}$) and $x_i\in (0,1)$ is independent of $m$ for all $i\in [d] \backslash \mathcal{J}$, we have, as $n\to \infty$,
    \begin{equation}\label{eq:thm:Leblanc.2012.boundary.Theorem.3}
        \BB(F_{n,m}^{\star}(\bb{x})) = m^{-1} B_{\mathcal{J}}(\bb{x}) + \OO_{\bb{\lambda}_{\mathcal{J}}}(m^{-3}) + \oo_{\bb{\lambda}_{\mathcal{J}}}(m^{-3/2}) \ind_{\{\mathcal{J}\neq [d]\}},
    \end{equation}
    where
    \begin{equation}\label{eq:def.BB.J.x}
        B_{\mathcal{J}}(\bb{x}) \leqdef \frac{1}{2} \sum_{i,j\in [d]\backslash \mathcal{J}} (x_i \ind_{\{i = j\}} - x_i x_j) \left. \frac{\partial^2}{\partial x_i \partial x_j} F(\bb{x})\right|_{\bb{x}_{\mathcal{J}} = \bb{0}} \hspace{-2mm} + \frac{1}{2 m} \sum_{i\in [d]} \lambda_i \left. \frac{\partial^2}{\partial x_i^2} F(\bb{x})\right|_{\bb{x} = \bb{0}}.
    \end{equation}
    In the case where $x_i = 0$ for some $i\in \mathcal{J}\neq \emptyset$, notice that $\BB(F_{n,m}^{\star}(\bb{x})) = 0$ because $F_{n,m}^{\star}(\bb{x}) = 0$ a.s.
\end{theorem}

Next, we obtain a general expression for the variance of the c.d.f.\ estimator on the simplex.

\begin{lemma}[Variance of $F_{n,m}^{\star}(\bb{x})$ on $\mathcal{S}_d$]\label{lem:Leblanc.2012.boundary.Lemma.8}
    Assume \eqref{eq:assump:F.cdf}.
    Then, uniformly for $\bb{x}\in \mathcal{S}_d$, we have, as $n\to \infty$,
    \begin{equation}\label{eq:lem:Leblanc.2012.boundary.Lemma.8}
        \begin{aligned}
            \VV(F_{n,m}^{\star}(\bb{x}))
            &= n^{-1} F(\bb{x}) (1 - F(\bb{x})) - n^{-1} \sum_{i\in [d]} \frac{\partial}{\partial x_i} F(\bb{x}) \hspace{-2mm} \sum_{\bb{k},\bb{\ell}\in \N_0^d \cap m\mathcal{S}_d} \hspace{-3mm} ((k_i \wedge \ell_i) / m - x_i) P_{\bb{k},m}(\bb{x}) P_{\bb{\ell},m}(\bb{x}) \\
            &\quad+ n^{-1} \sum_{i,j\in [d]} \OO\big(m^{-1} (x_i \ind_{\{i = j\}} - x_i x_j)\big) + \OO(n^{-1} m^{-2}).
        \end{aligned}
    \end{equation}
\end{lemma}

By combining Lemma~\ref{lem:Leblanc.2012.boundary.Lemma.7} and Lemma~\ref{lem:Leblanc.2012.boundary.Lemma.8}, we get the asymptotics of the variance of the c.d.f.\ estimator near the boundary.

\begin{theorem}[Variance of $F_{n,m}^{\star}(\bb{x})$ near the boundary of $\mathcal{S}_d$]\label{thm:Leblanc.2012.boundary.Theorem.4}
    Assume \eqref{eq:assump:F.cdf}.
    For any $\bb{x}\in \mathcal{S}_d$ such that $x_i = \lambda_i / m$ for all $i\in \mathcal{J}\subseteq [d]$ ($\lambda_i > 0$ is fixed $\forall i\in \mathcal{J}$) and $x_i\in (0,1)$ is independent of $m$ for all $i\in [d] \backslash \mathcal{J}$, we have, as $n\to \infty$,
    \begin{equation}\label{eq:thm:Leblanc.2012.boundary.Theorem.4}
        \begin{aligned}
            \VV(F_{n,m}^{\star}(\bb{x}))
            &= n^{-1} m^{-1} V_{\mathcal{J}}(\bb{x}) + n^{-1} F(\bb{x}) (1 - F(\bb{x})) \, \ind_{\{\mathcal{J} = \emptyset\}} \\
            &\quad+ \OO_{\bb{\lambda}_{\mathcal{J}}}\big(n^{-1} m^{-2}\big) + \oo_{\bb{\lambda}_{\mathcal{J}}\hspace{-0.2mm},\hspace{0.3mm} \bb{x}_{[d]\backslash \mathcal{J}}}\big(n^{-1} m^{-1/2}\big) \, \ind_{\{\mathcal{J} \neq [d]\}},
        \end{aligned}
    \end{equation}
    where
    \begin{equation}\label{eq:def.VV.J.x}
        V_{\mathcal{J}}(\bb{x}) \leqdef \sum_{i\in [d]} \big.\frac{\partial}{\partial x_i} F(\bb{x})\big|_{\bb{x}_{\mathcal{J}} = \bb{0}} \left\{\hspace{-1mm}
            \begin{array}{l}
                \lambda_i \, \big(1 - e^{-2 \lambda_i} (I_0(2 \lambda_i) + I_1(2\lambda_i))\big) \ind_{\{i\in \mathcal{J}\}} \\[1mm]
                - m^{1/2} \sqrt{\pi^{-1} x_i (1 - x_i)} \ind_{\{i\in [d]\backslash \mathcal{J}\}}
            \end{array}
            \hspace{-1mm}\right\}.
    \end{equation}
    In the case where $x_i = 0$ for some $i\in \mathcal{J}\neq \emptyset$, notice that $\VV(F_{n,m}^{\star}(\bb{x})) = 0$ because $F_{n,m}^{\star}(\bb{x}) = 0$ a.s.
\end{theorem}

By combining Theorem~\ref{thm:Leblanc.2012.boundary.Theorem.3} and Theorem~\ref{thm:Leblanc.2012.boundary.Theorem.4}, we get the asymptotics of the mean squared error of the c.d.f.\ estimator near the boundary.

\begin{corollary}[MSE of $F_{n,m}^{\star}(\bb{x})$ near the boundary of $\mathcal{S}_d$]
    Assume \eqref{eq:assump:F.cdf}.
    For any $\bb{x}\in \mathcal{S}_d$ such that $x_i = \lambda_i / m$ for all $i\in \mathcal{J}\subseteq [d]$ ($\lambda_i > 0$ is fixed $\forall i\in \mathcal{J}$) and $x_i\in (0,1)$ is independent of $m$ for all $i\in [d] \backslash \mathcal{J}$, we have, as $n\to \infty$,
    \begin{equation}
        \begin{aligned}
            \mathrm{MSE}(F_{n,m}^{\star}(\bb{x}))
            &=  n^{-1} m^{-1} V_{\mathcal{J}}(\bb{x}) + n^{-1} F(\bb{x}) (1 - F(\bb{x})) \, \ind_{\{\mathcal{J} = \emptyset\}} + m^{-2} B_{\mathcal{J}}^2(\bb{x}) \\
            &\quad+ \OO_{\bb{\lambda}_{\mathcal{J}}}\big(n^{-1} m^{-2}\big) + \oo_{\bb{\lambda}_{\mathcal{J}}\hspace{-0.2mm},\hspace{0.3mm} \bb{x}_{[d]\backslash \mathcal{J}}}\big(n^{-1} m^{-1/2}\big) \, \ind_{\{\mathcal{J} \neq [d]\}} \\
            &\quad+ \OO_{\bb{\lambda}_{\mathcal{J}}}(m^{-4}) + \oo_{\bb{\lambda}_{\mathcal{J}}}(m^{-5/2}) \ind_{\{\mathcal{J}\neq [d]\}},
        \end{aligned}
    \end{equation}
    where $B_{\mathcal{J}}(\bb{x})$ is defined in \eqref{eq:def.BB.J.x} and $V_{\mathcal{J}}(\bb{x})$ is defined in \eqref{eq:def.VV.J.x}.
    In the case where $x_i = 0$ for some $i\in \mathcal{J}\neq \emptyset$, notice that $\mathrm{MSE}(F_{n,m}^{\star}(\bb{x})) = 0$ because $F_{n,m}^{\star}(\bb{x}) = 0$ a.s.
    Furthermore, as pointed out by \cite[p.2772]{MR2925964} for $d=1$, there is no optimal $m$ with respect to the MSE when $\mathcal{J} \neq \emptyset$.
    This is also true here.
    The remaining case $\mathcal{J} = \emptyset$ (i.e., when $\bb{x}$ is far from the boundary in every component) was already treated in Corollary~1 of \cite{MR4287788}.
\end{corollary}

\section{Proof of the results for the density estimator \texorpdfstring{$\widetilde{f}_{n,m}$}{hat(f)n,m}}\label{sec:proofs.results.density.estimator}

\subsection{Proof of Lemma~\ref{lem:Leblanc.2012.boundary.Lemma.2}}

    Take $\delta_n \searrow 0$ slow enough as $n\to \infty$ (for example $\delta_n \geq m^{-1/4}$) that standard concentration bounds for the binomial distribution yield
    \begin{equation}\label{eq:contribution.outside.bulk}
        \sum_{\substack{\bb{k}\in \mathbb{N}_0^d \cap m \mathcal{S}_d \\ \|\bb{k} / m - \bb{x}\|_1 > \delta_n}} \hspace{-3mm}P_{\bb{k},m}(\bb{x}) \leq \sum_{\ell\in [d]} \sum_{\substack{\bb{k}\in \mathbb{N}_0^d \cap m \mathcal{S}_d \\ |k_{\ell} / m - x_{\ell}| > \delta_n / d}} \hspace{-3mm}P_{\bb{k},m}(\bb{x})\leq C_d e^{-c_d \delta_n^2 m} = \oo(m^{-2}),
    \end{equation}
    for some appropriate constants $c_d,C_d > 0$.
    For any $\bb{k}$ such that $\|\bb{k} / m - \bb{x}\|_1 \leq \delta_n$, we can use Taylor expansions and our assumption that $f$ is twice continuously differentiable to obtain
    \begin{align*}\label{thm:Theorem.3.2.and.3.3.Babu.Canty.Chaubey.second.asymp.begin.1}
        &m^d \int_{(\frac{\bb{k}}{m}, \frac{\bb{k} + 1}{m}]} \hspace{-0.5mm} f(\bb{y}) \rd \bb{y} - f(\bb{x}) \notag \\
        &= f(\bb{k} / m) - f(\bb{x}) + \frac{1}{2m} \sum_{i\in [d]} \frac{\partial}{\partial x_i} f(\bb{k} / m) + \frac{1}{m^2} \sum_{i,j\in [d]} \big(\tfrac{1}{6} \ind_{\{i = j\}} + \tfrac{1}{8} \ind_{\{i \neq j\}}\big) \, \frac{\partial^2}{\partial x_i \partial x_j} f(\bb{k} / m) (1 + \oo(1)) \notag \\
        &= \frac{1}{m} \sum_{i\in [d]} (k_i - m x_i) \, \frac{\partial}{\partial x_i} f(\bb{x}) + \frac{1}{2 m^2} \sum_{i,j\in [d]} (k_i - m x_i) (k_j - m x_j) \, \frac{\partial^2}{\partial x_i \partial x_j} f(\bb{x}) (1 + \oo(1)) \notag \\
        &\quad+ \frac{1}{2m} \sum_{i\in [d]} \frac{\partial}{\partial x_i} f(\bb{x}) + \frac{1}{2 m^2} \sum_{i,j\in [d]} (k_j - m x_j) \, \frac{\partial^2}{\partial x_i \partial x_j} f(\bb{x}) (1 + \oo(1)) \notag \\
        &\quad+ \frac{1}{m^2} \sum_{i,j\in [d]} \big(\tfrac{1}{6} \ind_{\{i = j\}} + \tfrac{1}{8} \ind_{\{i \neq j\}}\big) \, \frac{\partial^2}{\partial x_i \partial x_j} f(\bb{x}) (1 + \oo(1)) \notag \\
        &= \frac{1}{m} \sum_{i\in [d]} (k_i - (m - 1) x_i) \, \frac{\partial}{\partial x_i} f(\bb{x}) + \frac{1}{m} \sum_{i\in [d]} \big(\tfrac{1}{2} - x_i\big) \, \frac{\partial}{\partial x_i} f(\bb{x}) \\
        &\quad+ \frac{1}{2 m^2} \sum_{i,j\in [d]} (k_i - (m - 1) x_i) (k_j - (m - 1) x_j) \, \frac{\partial^2}{\partial x_i \partial x_j} f(\bb{x}) \notag \\
        &\quad- \frac{1}{2 m^2} \sum_{i,j\in [d]} \big[2 x_i \, (k_j - (m - 1) x_j) - x_i x_j\big] \, \frac{\partial^2}{\partial x_i \partial x_j} f(\bb{x}) + \frac{1}{2 m^2} \sum_{i,j\in [d]} (k_j - (m - 1) x_j) \, \frac{\partial^2}{\partial x_i \partial x_j} f(\bb{x}) \notag \\
        &\quad+ \frac{1}{m^2} \sum_{i,j\in [d]} \big(\tfrac{1}{6} \ind_{\{i = j\}} + \tfrac{1}{8} \ind_{\{i \neq j\}} - \tfrac{1}{2} x_j\big) \, \frac{\partial^2}{\partial x_i \partial x_j} f(\bb{x}) \\
        &\quad+ \frac{1}{m^2} \sum_{i\in [d]} \oo\big(1 + |k_i - (m - 1) x_i| + |k_i - (m - 1) x_i|^2\big).
    \end{align*}
    If we multiply the last expression by $P_{\bb{k},m-1}(\bb{x})$ and sum over $\bb{k}\in \N_0^d \cap (m-1)\mathcal{S}_d$, then the joint moments from Lemma~ \ref{lem:Leblanc.2012.boundary.Lemma.1} and Jensen's inequality yield
    \begin{equation}\label{thm:Theorem.3.2.and.3.3.Babu.Canty.Chaubey.second.asymp.end}
        \begin{aligned}
            \EE[\widetilde{f}_{n,m}(\bb{x})] - f(\bb{x})
            &= \frac{1}{m} \sum_{i\in [d]} \big(\tfrac{1}{2} - x_i\big) \, \frac{\partial}{\partial x_i} f(\bb{x}) + \frac{m-1}{2 m^2} \sum_{i,j\in [d]} \big(x_i \ind_{\{i = j\}} - x_i x_j\big) \, \frac{\partial^2}{\partial x_i \partial x_j} f(\bb{x}) \\
            &\quad+ \frac{1}{m^2} \sum_{i,j\in [d]} \left\{\tfrac{1}{6} \ind_{\{i = j\}} + \tfrac{1}{8} \ind_{\{i \neq j\}} - \tfrac{1}{2} x_j + \tfrac{1}{2} x_i x_j\right\} \frac{\partial^2}{\partial x_i \partial x_j} f(\bb{x}) \\
            &\quad+ \frac{1}{m^2} \sum_{i\in [d]} \oo\left(1 + \sqrt{(m-1) x_i (1 - x_i)} + (m - 1) x_i (1 - x_i)\right).
        \end{aligned}
    \end{equation}
    The conclusion follows.

\subsection{Proof of Theorem~\ref{thm:Leblanc.2012.boundary.Theorem.1}}

    Take $\bb{x}$ as in the statement of the theorem.
    Using the notation from \eqref{eq:def:b.x}, we have
    \begin{equation}\label{eq:thm:Leblanc.2012.boundary.Theorem.1.eq.Delta.1}
        \begin{aligned}
            \Delta_1(\bb{x})
            &= \sum_{i\in \mathcal{J}} \big(\tfrac{1}{2} - \tfrac{\lambda_i}{m}\big) \, \frac{\partial}{\partial x_i} f(\bb{x}) + \sum_{i\in \mathcal{J}} \frac{\lambda_i}{2m} \frac{\partial^2}{\partial x_i^2} f(\bb{x}) + \hspace{-1mm} \sum_{i\in [d]\backslash \mathcal{J}} \big(\tfrac{1}{2} - x_i\big) \, \frac{\partial}{\partial x_i} f(\bb{x}) - \hspace{-1mm} \sum_{\substack{i\in \mathcal{J} \\ j\in [d]\backslash \mathcal{J}}} \frac{\lambda_i x_j}{m} \frac{\partial^2}{\partial x_i \partial x_j} f(\bb{x}) \\
            &\quad+ \sum_{i,j\in [d]\backslash \mathcal{J}} \tfrac{1}{2} \big(x_i \ind_{\{i = j\}} - x_i x_j\big) \, \frac{\partial^2}{\partial x_i \partial x_j} f(\bb{x}) + \OO_{\bb{\lambda}_{\mathcal{J}}}(m^{-2}),
        \end{aligned}
    \end{equation}
    and
    \begin{equation}\label{eq:thm:Leblanc.2012.boundary.Theorem.1.eq.Delta.2}
        \begin{aligned}
            \Delta_2(\bb{x})
            &= \sum_{(i,j)\in [d]^2\backslash ([d]\backslash \mathcal{J})^2} \big(\tfrac{1}{6} \ind_{\{i = j\}} + \tfrac{1}{8} \ind_{\{i \neq j\}}\big) \, \frac{\partial^2}{\partial x_i \partial x_j} f(\bb{x}) \\
            &\quad+ \sum_{i,j\in [d]\backslash \mathcal{J}} \left\{\tfrac{1}{6} \ind_{\{i = j\}} + \tfrac{1}{8} \ind_{\{i \neq j\}} - \tfrac{1}{2} x_i \ind_{\{i = j\}} - \tfrac{1}{2} x_j + x_i x_j\right\} \frac{\partial^2}{\partial x_i \partial x_j} f(\bb{x}) + \OO_{\bb{\lambda}_{\mathcal{J}}}(m^{-1}).
        \end{aligned}
    \end{equation}
    For all $i,j\in [d]$, note that
    \begin{align}\label{eq:thm:Leblanc.2012.boundary.Theorem.1.eq.Taylor}
        \begin{aligned}
            \frac{\partial}{\partial x_i} f(\bb{x})
            &= \left.\frac{\partial}{\partial x_i} f(\bb{x})\right|_{\bb{x}_{\mathcal{J}} = \bb{0}} + \sum_{j\in \mathcal{J}} \frac{\lambda_j}{m} \left.\frac{\partial^2}{\partial x_i \partial x_j} f(\bb{x})\right|_{\bb{x}_{\mathcal{J}} = \bb{0}} (1 + \oo_{\bb{\lambda}_{\mathcal{J}}}(1)), \\
            %%%
            \frac{\partial^2}{\partial x_i \partial x_j} f(\bb{x})
            &= \left.\frac{\partial^2}{\partial x_i \partial x_j} f(\bb{x})\right|_{\bb{x}_{\mathcal{J}} = \bb{0}} (1 + \oo_{\bb{\lambda}_{\mathcal{J}}}(1)),
        \end{aligned}
    \end{align}
    Then, from \eqref{eq:thm:Leblanc.2012.boundary.Theorem.1.eq.Delta.1}, \eqref{eq:thm:Leblanc.2012.boundary.Theorem.1.eq.Delta.2} \eqref{eq:thm:Leblanc.2012.boundary.Theorem.1.eq.Taylor} and Lemma~\ref{lem:Leblanc.2012.boundary.Lemma.2}, we can easily deduce the conclusion.

\subsection{Proof of Lemma~\ref{lem:Leblanc.2012.boundary.Lemma.4}}

    By the independence of the observations $\bb{X}_1,\bb{X}_2,\dots,\bb{X}_n$, we have
    \begin{equation}\label{eq:thm:Leblanc.2012.boundary.Theorem.2.eq.1}
        \VV(\widetilde{f}_{n,m}(\bb{x})) = n^{-1} \left\{m^{2d} \hspace{-3mm}\sum_{\bb{k}\in \N_0^d \cap (m-1)\mathcal{S}_d} \int_{(\frac{\bb{k}}{m}, \frac{\bb{k} + 1}{m}]} \hspace{-0.5mm} f(\bb{y}) \rd \bb{y} \, P_{\bb{k},m-1}^2(\bb{x}) - \EE[\widetilde{f}_{n,m}(\bb{x})]^2\right\}.
    \end{equation}
    From Lemma~\ref{lem:Leblanc.2012.boundary.Lemma.2}, we already know that $\EE[\widetilde{f}_{n,m}(\bb{x})] = f(\bb{x}) + \OO(m^{-1})$, uniformly for $\bb{x}\in \mathcal{S}_d$.
    We can also expand the integral using a Taylor expansion:
    \begin{equation}\label{eq:thm:Leblanc.2012.boundary.Theorem.2.eq.2}
        \begin{aligned}
            &m^d \hspace{-3mm}\sum_{\bb{k}\in \N_0^d \cap (m-1)\mathcal{S}_d} \int_{(\frac{\bb{k}}{m}, \frac{\bb{k} + 1}{m}]} \hspace{-0.5mm} f(\bb{y}) \rd \bb{y} \, P_{\bb{k},m-1}^2(\bb{x}) \\
            &\quad= f(\bb{x}) \hspace{-3mm} \sum_{\bb{k}\in \N_0^d \cap (m-1)\mathcal{S}_d} \hspace{-3mm} P_{\bb{k},m-1}^2(\bb{x}) + \frac{1}{m} \sum_{i\in [d]} \OO\left(\sum_{\bb{k}\in \N_0^d \cap (m-1)\mathcal{S}_d} \hspace{-3mm} |k_i - m x_i| P_{\bb{k},m-1}^2(\bb{x})\right) + \OO(m^{-1}).
        \end{aligned}
    \end{equation}
    The Cauchy-Schwarz inequality yields
    \begin{equation}\label{eq:thm:Leblanc.2012.boundary.Theorem.2.eq.3}
        \begin{aligned}
            &\sum_{\bb{k}\in \N_0^d \cap (m-1)\mathcal{S}_d} \hspace{-3mm} |k_i - m x_i| P_{\bb{k},m-1}^2(\bb{x}) \\
            &\qquad\leq \sqrt{\sum_{\bb{k}\in \N_0^d \cap (m-1)\mathcal{S}_d} \hspace{-3mm} |k_i - m x_i|^2 P_{\bb{k},m-1}(\bb{x})} \sqrt{\sum_{\bb{k}\in \N_0^d \cap (m-1)\mathcal{S}_d} \hspace{-3mm} P_{\bb{k},m-1}^3(\bb{x})}.
        \end{aligned}
    \end{equation}
    By putting \eqref{eq:thm:Leblanc.2012.boundary.Theorem.2.eq.1}, \eqref{eq:thm:Leblanc.2012.boundary.Theorem.2.eq.2} and \eqref{eq:thm:Leblanc.2012.boundary.Theorem.2.eq.3} together, we get the conclusion.

\subsection{Proof of Theorem~\ref{thm:Leblanc.2012.boundary.Theorem.2}}

    By Lemma~\ref{lem:Leblanc.2012.boundary.Lemma.4}, Lemma~\ref{lem:Leblanc.2012.boundary.Lemma.3}, and \eqref{eq:thm:central.moments.eq.2} in Lemma~ \ref{lem:Leblanc.2012.boundary.Lemma.1}, we have
    \begin{equation}
        \begin{aligned}
            \VV(\widetilde{f}_{n,m}(\bb{x}))
            &= \frac{m^{(d + |\mathcal{J}|)/2}}{n} f(\bb{x}) \left\{\psi_{[d]\backslash \mathcal{J}}(\bb{x}) \prod_{i\in \mathcal{J}} e^{-2 \lambda_i} I_0(2 \lambda_i) + \OO_{\bb{\lambda}_{\mathcal{J}}}(m^{-1}) + \oo_{\bb{x}_{[d]\backslash \mathcal{J}}}(1) \, \ind_{\{\mathcal{J} \neq [d]\}}\right\} \\
            &\quad+ \frac{m^{d-1}}{n} \sum_{i\in [d]} \left\{\OO_{\bb{\lambda}_{\mathcal{J}}}(1 + \ind_{\{i\not\in \mathcal{J}\}} m^{1/2}) \, \OO_{\bb{\lambda}_{\mathcal{J}}\hspace{-0.2mm},\hspace{0.3mm} \bb{x}_{[d]\backslash \mathcal{J}}}\big(m^{-(d - |\mathcal{J}|)/2}\big) + \OO(1)\right\}.
        \end{aligned}
    \end{equation}
    Using the Taylor expansion $f(\bb{x}) = \left.f(\bb{x})\right|_{\bb{x}_{\mathcal{J}} = \bb{0}} + \OO_{\bb{\lambda}_{\mathcal{J}}}(m^{-1})$, we get
    \begin{equation}
        \begin{aligned}
            \VV(\widetilde{f}_{n,m}(\bb{x}))
            &= \frac{m^{(d + |\mathcal{J}|)/2}}{n} \left(\left.f(\bb{x})\right|_{\bb{x}_{\mathcal{J}} = \bb{0}} + \OO_{\bb{\lambda}_{\mathcal{J}}}(m^{-1})\right) \\
            &\qquad\cdot \left\{\psi_{[d]\backslash \mathcal{J}}(\bb{x}) \prod_{i\in \mathcal{J}} e^{-2 \lambda_i} I_0(2 \lambda_i) + \OO_{\bb{\lambda}_{\mathcal{J}}}(m^{-1}) + \oo_{\bb{x}_{[d]\backslash \mathcal{J}}}(1) \, \ind_{\{\mathcal{J} \neq [d]\}}\right\} \\
            &\quad+ \frac{m^{(d + |\mathcal{J}|)/2}}{n} \, \OO_{\bb{\lambda}_{\mathcal{J}}\hspace{-0.2mm},\hspace{0.3mm} \bb{x}_{[d]\backslash \mathcal{J}}}\big(m^{-1} + \ind_{\{\mathcal{J} \neq [d]\}} m^{-1/2}\big).
        \end{aligned}
    \end{equation}
    The conclusion follows.

\section{Proof of the results for the c.d.f.\ estimator \texorpdfstring{$F_{n,m}^{\star}$}{Fn,m star}}\label{sec:proofs.results.cdf.estimator}

\subsection{Proof of Lemma~\ref{lem:Leblanc.2012.boundary.Lemma.6}}

    Take $\delta_n \searrow 0$ slow enough as $n\to \infty$ that the contribution coming from points $\bb{k} / m$ outside the bulk of the multinomial distribution (i.e., $\|\bb{k} / m - \bb{x}\|_1 > \delta_n$) is negligible, exactly as we did in \eqref{eq:contribution.outside.bulk}.
    Then, for any $\bb{k}$ such that $\|\bb{k} / m - \bb{x}\|_1 \leq \delta_n$, we can use a Taylor expansion and our assumption that $F$ is three-times continuously differentiable to obtain
    \begin{equation}\label{eq:lem:Leblanc.2012.boundary.Lemma.6.eq.Taylor}
        \begin{aligned}
            F(\bb{k} / m)
            &= F(\bb{x}) + \sum_{i\in [d]} (k_i / m - x_i) \, \frac{\partial}{\partial x_i} F(\bb{x}) + \frac{1}{2} \sum_{i,j\in [d]} (k_i / m - x_i) (k_j / m - x_j) \, \frac{\partial^2}{\partial x_i \partial x_j} F(\bb{x}) \\
            &\quad+ \frac{1}{6} \sum_{i,j,\ell\in [d]} (k_i / m - x_i) (k_j / m - x_j) (k_{\ell} / m - x_{\ell}) \, \frac{\partial^3}{\partial x_i \partial x_j \partial x_{\ell}} F(\bb{x}) + \oo\big(\|\bb{k} / m - \bb{x}\|_1^3\big).
        \end{aligned}
    \end{equation}
    If we multiply the last expression by $P_{\bb{k},m}(\bb{x})$, sum over $\bb{k}\in \N_0^d \cap m \mathcal{S}_d$, and then take the expectation on both sides, we get
    \begin{equation}
        \begin{aligned}
            \EE[F_{n,m}^{\star}(\bb{x})]
            &= F(\bb{x}) + \sum_{i\in [d]} \EE\big[\xi_i / m - x_i\big] \, \frac{\partial}{\partial x_i} F(\bb{x}) + \frac{1}{2} \sum_{i,j\in [d]} \EE\big[(\xi_i / m - x_i) (\xi_j / m - x_j)\big] \, \frac{\partial^2}{\partial x_i \partial x_j} F(\bb{x}) \\
            &\quad+ \frac{1}{6} \sum_{i,j,\ell\in [d]} \EE\big[(\xi_i / m - x_i) (\xi_j / m - x_j) (\xi_{\ell} / m - x_{\ell})\big] \, \frac{\partial^3}{\partial x_i \partial x_j \partial x_{\ell}} F(\bb{x}) \\
            &\quad+ \sum_{i,j,\ell\in [d]} \oo\left(\EE\big[|\xi_i / m - x_i| |\xi_j / m - x_j| |\xi_{\ell} / m - x_{\ell}|\big]\right),
        \end{aligned}
    \end{equation}
    where $\bb{\xi} = (\xi_1,\xi_2,\dots,\xi_d)\sim \mathrm{Multinomial}\hspace{0.2mm}(m,\bb{x})$.
    From the multinomial joint central moments in Lemma~\ref{lem:Leblanc.2012.boundary.Lemma.1}, we get
    \begin{align}
        \EE[F_{n,m}^{\star}(\bb{x})]
        &= F(\bb{x}) + \frac{1}{2m} \sum_{i,j\in [d]} (x_i \ind_{\{i = j\}} - x_i x_j) \, \frac{\partial^2}{\partial x_i \partial x_j} F(\bb{x}) \notag \\
        &\quad+ \frac{1}{6m^2} \sum_{i,j,\ell\in [d]} \left(\hspace{-1mm}
            \begin{array}{l}
                2 x_i x_j x_{\ell} - \ind_{\{i = j\}} x_i x_{\ell} - \ind_{\{j = \ell\}} x_i x_j \\
                - \ind_{\{i = \ell\}} x_j x_{\ell} + \ind_{\{i = j = \ell\}} x_i
            \end{array}
            \hspace{-1mm}\right) \frac{\partial^3}{\partial x_i \partial x_j \partial x_{\ell}} F(\bb{x})
            \notag \\
        &\quad+ \frac{1}{m^3} \sum_{i,j,\ell\in [d]} \oo\left(\EE\big[|\xi_i - m \, x_i| |\xi_j - m \, x_j| |\xi_{\ell} - m \, x_{\ell}|\big]\right).
    \end{align}
    We apply Holder's inequality on the error term to get the conclusion.

\subsection{Proof of Theorem~\ref{thm:Leblanc.2012.boundary.Theorem.3}}

    Take $\bb{x}$ as in the statement of the theorem.
    For all $i,j,\ell\in [d]$, note that
    \begin{align}\label{eq:thm:Leblanc.2012.boundary.Theorem.3.eq.Taylor}
        \frac{\partial^2}{\partial x_i \partial x_j} F(\bb{x})
        &= \left. \frac{\partial^2}{\partial x_i \partial x_j} F(\bb{x})\right|_{\bb{x}_{\mathcal{J}} = \bb{0}} + \sum_{\ell\in \mathcal{J}} \frac{\lambda_{\ell}}{m} \left. \frac{\partial^3}{\partial x_i \partial x_j \partial x_{\ell}} F(\bb{x})\right|_{\bb{x}_{\mathcal{J}} = \bb{0}} (1 + \oo_{\bb{\lambda}_{\mathcal{J}}}(1)), \\
        %%%
        \frac{\partial^3}{\partial x_i \partial x_j \partial x_{\ell}} F(\bb{x})
        &= \left. \frac{\partial^3}{\partial x_i \partial x_j \partial x_{\ell}} F(\bb{x})\right|_{\bb{x}_{\mathcal{J}} = \bb{0}} + \oo_{\bb{\lambda}_{\mathcal{J}}}(1).
    \end{align}
    By Lemma~\ref{lem:Leblanc.2012.boundary.Lemma.6} and the second and fourth moments in Lemma~\ref{lem:Leblanc.2012.boundary.Lemma.1}, we deduce
    \begin{align}
        &\BB(F_{n,m}^{\star}(\bb{x})) \\
        &\quad= \frac{1}{m} \sum_{i,j\in [d]\backslash \mathcal{J}} \tfrac{1}{2} (x_i \ind_{\{i = j\}} - x_i x_j) \left. \frac{\partial^2}{\partial x_i \partial x_j} F(\bb{x})\right|_{\bb{x}_{\mathcal{J}} = \bb{0}} \notag \\
        &\qquad+ \frac{1}{m^2} \left\{\hspace{-1mm}
            \begin{array}{l}
                \frac{1}{6} \sum_{i,j,\ell\in [d]\backslash \mathcal{J}}
                \left(\hspace{-1mm}
                    \begin{array}{l}
                        2 x_i x_j x_{\ell} - \ind_{\{i = j\}} x_i x_{\ell} - \ind_{\{j = \ell\}} x_i x_j \\
                        - \ind_{\{i = \ell\}} x_j x_{\ell} + \ind_{\{i = j = \ell\}} x_i
                    \end{array}
                    \hspace{-1mm}\right) \left. \frac{\partial^3}{\partial x_i \partial x_j \partial x_{\ell}} F(\bb{x})\right|_{\bb{x}_{\mathcal{J}} = \bb{0}} \\
                + \sum_{i,j\in [d]\backslash \mathcal{J}} \sum_{\ell\in \mathcal{J}} \tfrac{\lambda_{\ell}}{2} (x_i \ind_{\{i = j\}} - x_i x_j) \left. \frac{\partial^3}{\partial x_i \partial x_j \partial x_{\ell}} F(\bb{x})\right|_{\bb{x}_{\mathcal{J}} = \bb{0}} \hspace{-2mm} \\
                + \frac{1}{2} \sum_{i\in \mathcal{J}} \sum_{j\in [d]\backslash \mathcal{J}} \big(\lambda_i \ind_{\{i = j\}} - 2 \lambda_i x_j\big) \left. \frac{\partial^2}{\partial x_i \partial x_j} F(\bb{x})\right|_{\bb{x}_{\mathcal{J}} = \bb{0}}
            \end{array}
            \hspace{-1mm}\right\} \notag \\
        &\qquad+ \OO_{\bb{\lambda}_{\mathcal{J}}}(m^{-3}) + \oo_{\bb{\lambda}_{\mathcal{J}}}(m^{-3/2}) \ind_{\{\mathcal{J}\neq [d]\}},
    \end{align}
    This ends the proof.

\subsection{Proof of Lemma~\ref{lem:Leblanc.2012.boundary.Lemma.8}}

    To estimate the variance of $F_{n,m}^{\star}(\bb{x})$, note that
    \begin{equation}\label{eq:hat.F.as.mean.Z.i.m}
        F_{n,m}^{\star}(\bb{x}) - \EE[F_{n,m}^{\star}(\bb{x})] = \frac{1}{n} \sum_{i=1}^n \sum_{\bb{k}\in \N_0^d \cap m\mathcal{S}_d} \big(\ind_{(-\bb{\infty},\frac{\bb{k}}{m}]}(\bb{X}_i) - F(\bb{k} / m)\big) P_{\bb{k},m}(\bb{x}).
    \end{equation}
    By the independence of the observations $\bb{X}_1,\bb{X}_2,\dots,\bb{X}_n$, we get
    \begin{equation}
        \VV(F_{n,m}^{\star}(\bb{x})) = n^{-1} \hspace{0.2mm} \left\{\sum_{\bb{k},\bb{\ell}\in \N_0^d \cap m\mathcal{S}_d} \hspace{-3mm} F((\bb{k} \wedge \bb{\ell}) / m) P_{\bb{k},m}(\bb{x}) P_{\bb{\ell},m}(\bb{x}) - \EE[F_{n,m}^{\star}(\bb{x})]^2\right\}.
    \end{equation}
    Using the expansion in \eqref{eq:lem:Leblanc.2012.boundary.Lemma.6.eq.Taylor} and Lemma~\ref{lem:Leblanc.2012.boundary.Lemma.6}, the above is
    \begin{equation}\label{eq:thm:bias.var.Taylor.expansion}
        \begin{aligned}
            = n^{-1} \left\{\hspace{-1mm}
                \begin{array}{l}
                    F(\bb{x}) - F^2(\bb{x}) + \sum_{i,j\in [d]} \OO\big(m^{-1} (x_i \ind_{\{i = j\}} - x_i x_j)\big) + \OO(m^{-2}) \\
                    + \sum_{i\in [d]} \frac{\partial}{\partial x_i} F(\bb{x}) \sum_{\bb{k},\bb{\ell}\in \N_0^d \cap m\mathcal{S}_d} ((k_i \wedge \ell_i) / m - x_i) P_{\bb{k},m}(\bb{x}) P_{\bb{\ell},m}(\bb{x}) \\
                    + \sum_{i\in [d]} \OO\left(m^{-1} \sum_{\bb{k}\in \N_0^d \cap m\mathcal{S}_d} |k_i - m \, x_i|^2 P_{\bb{k},m}(\bb{x})\right)
                \end{array}
                \hspace{-1mm}\right\}.
        \end{aligned}
    \end{equation}
    By the second moment expression in Lemma~\ref{lem:Leblanc.2012.boundary.Lemma.1}, we get the conclusion.

\subsection{Proof of Theorem~\ref{thm:Leblanc.2012.boundary.Theorem.4}}

    By Lemma~\ref{lem:Leblanc.2012.boundary.Lemma.8}, Lemma~\ref{lem:Leblanc.2012.boundary.Lemma.7} and \eqref{eq:thm:central.moments.eq.2} in Lemma~ \ref{lem:Leblanc.2012.boundary.Lemma.1}, we have
    \begin{equation}
        \begin{aligned}
            \VV(F_{n,m}^{\star}(\bb{x}))
            &= n^{-1} m^{-1} \sum_{i\in [d]} \frac{\partial}{\partial x_i} F(\bb{x}) \, \left\{\hspace{-1mm}
                \begin{array}{l}
                    - \lambda_i e^{-2 \lambda_i} (I_0(2 \lambda_i) + I_1(2\lambda_i)) \ind_{\{i\in \mathcal{J}\}} \\
                    - m^{1/2} \sqrt{\pi^{-1} x_i (1 - x_i)} \ind_{\{i\in [d]\backslash \mathcal{J}\}}
                \end{array}
                \hspace{-1mm}\right\} \\
            &\quad+ n^{-1} F(\bb{x}) (1 - F(\bb{x})) + \OO_{\bb{\lambda}_{\mathcal{J}}}\big(n^{-1} m^{-2}\big) \\
            &\quad+ \oo_{\bb{\lambda}_{\mathcal{J}}\hspace{-0.2mm},\hspace{0.3mm} \bb{x}_{[d]\backslash \mathcal{J}}}\big(n^{-1} m^{-1/2}\big) \, \ind_{\{\mathcal{J} \neq [d]\}}.
        \end{aligned}
    \end{equation}
    Now, using the fact that
    \begin{align}\label{eq:thm:Leblanc.2012.boundary.Theorem.4.eq.Taylor}
        F(\bb{x})
        &= \underbrace{\big.F(\bb{x})\big|_{\bb{x}_{\mathcal{J}} = \bb{0}}}_{=~0~\text{when } \mathcal{J}\neq \emptyset} + \sum_{i\in \mathcal{J}} \frac{\lambda_i}{m} \big.\frac{\partial}{\partial x_i} F(\bb{x})\big|_{\bb{x}_{\mathcal{J}} = \bb{0}} + \OO_{\bb{\lambda}_{\mathcal{J}}}(m^{-2}), \\
        %%%
        \frac{\partial}{\partial x_i} F(\bb{x})
        &= \big.\frac{\partial}{\partial x_i} F(\bb{x})\big|_{\bb{x}_{\mathcal{J}} = \bb{0}} + \sum_{j\in \mathcal{J}} \frac{\lambda_j}{m} \big.\frac{\partial^2}{\partial x_i \partial x_j} F(\bb{x})\big|_{\bb{x}_{\mathcal{J}} = \bb{0}} + \OO_{\bb{\lambda}_{\mathcal{J}}}(m^{-2}),
    \end{align}
    we get
    \begin{equation}
        \begin{aligned}
            \VV(F_{n,m}^{\star}(\bb{x}))
            &= n^{-1} m^{-1} \sum_{i\in [d]} \big.\frac{\partial}{\partial x_i} F(\bb{x})\big|_{\bb{x}_{\mathcal{J}} = \bb{0}} \, \left\{\hspace{-1mm}
                \begin{array}{l}
                    \lambda_i \, \big(1 - e^{-2 \lambda_i} (I_0(2 \lambda_i) + I_1(2\lambda_i))\big) \ind_{\{i\in \mathcal{J}\}} \\
                    - m^{1/2} \sqrt{\pi^{-1} x_i (1 - x_i)} \ind_{\{i\in [d]\backslash \mathcal{J}\}}
                \end{array}
                \hspace{-1mm}\right\} \\
            &\quad+ n^{-1} F(\bb{x}) (1 - F(\bb{x})) \, \ind_{\{\mathcal{J} = \emptyset\}} + \OO_{\bb{\lambda}_{\mathcal{J}}}\big(n^{-1} m^{-2}\big) + \oo_{\bb{\lambda}_{\mathcal{J}}\hspace{-0.2mm},\hspace{0.3mm} \bb{x}_{[d]\backslash \mathcal{J}}}\big(n^{-1} m^{-1/2}\big) \, \ind_{\{\mathcal{J} \neq [d]\}}.
        \end{aligned}
    \end{equation}
    This ends the proof.

\appendix

\section{Technical lemmas}\label{sec:tech.lemmas}

The first lemma is a generalization of Lemma~3 in \cite{MR4287788}.
It is used in the proof of Theorem~\ref{thm:Leblanc.2012.boundary.Theorem.2}.

\begin{lemma}\label{lem:Leblanc.2012.boundary.Lemma.3}
    For any $\bb{x}\in \mathcal{S}_d$ such that $x_i = \lambda_i / m$ for all $i\in \mathcal{J}\subseteq [d]$ ($\lambda_i\geq 0$ is fixed $\forall i\in \mathcal{J}$) and $x_i\in (0,1)$ independent of $m$ for all $i\in [d] \backslash \mathcal{J}$, we have, as $m\to \infty$,
    \begin{align}
        m^{(d - |\mathcal{J}|)/2} \hspace{-3mm} \sum_{\bb{k}\in \N_0^d \cap (m - 1)\mathcal{S}_d} \hspace{-3mm} P_{\bb{k},m-1}^2(\bb{x})
        &= \psi_{[d]\backslash \mathcal{J}}(\bb{x}) \prod_{i\in \mathcal{J}} e^{-2 \lambda_i} I_0(2 \lambda_i) \\[-2mm]
        &\hspace{25mm}+ \OO_{\bb{\lambda}_{\mathcal{J}}}(m^{-1}) + \oo_{\bb{\lambda}_{\mathcal{J}}\hspace{-0.2mm},\hspace{0.3mm} \bb{x}_{[d]\backslash \mathcal{J}}}(1) \, \ind_{\{\mathcal{J} \neq [d]\}}, \label{eq:lem:Leblanc.2012.boundary.Lemma.3} \\[2mm]
        \sum_{\bb{k}\in \N_0^d \cap (m - 1)\mathcal{S}_d} P_{\bb{k},m-1}^3(\bb{x})
        &= \OO_{\bb{\lambda}_{\mathcal{J}}\hspace{-0.2mm},\hspace{0.3mm} \bb{x}_{[d]\backslash \mathcal{J}}}\big(m^{-(d - |\mathcal{J}|)}\big), \label{eq:lem:technical.sums.power.3.eq.2}
    \end{align}
    where $I_0$ is defined in \eqref{eq:def:I.0.I.1} and $\psi_{[d]\backslash \mathcal{J}}$ is defined in \eqref{eq:vector.subset}.
\end{lemma}

The crucial tool for the proof is a local limit theorem for the multinomial distribution in \cite{MR0478288} that combines the multivariate Poisson approximation and the multivariate normal approximation, depending on which components of $\bb{x}\in \mathcal{S}_d$ are close to the boundary.

\begin{proof}[\bf Proof of Lemma~\ref{lem:Leblanc.2012.boundary.Lemma.3}]
    Denote $\bb{x}_A \leqdef (x_i)_{i\in A}$ for any subset of indices $A\subseteq [d]$, and let
    \begin{equation}
        \phi_{\emptyset}(\bb{x}) \leqdef 1, \qquad \phi_A(\bb{x}) \leqdef \frac{\exp\big(-\frac{1}{2} \bb{\delta}_{\bb{x}_A}^{\top} (\mathrm{diag}(\bb{x}_A) - \bb{x}_A \bb{x}_A^{\top})^{-1} \bb{\delta}_{\bb{x}_A}\big)}{\sqrt{(2\pi)^{|A|} \, (1 - \sum_{i\in A} x_i) \prod_{i\in A} x_i}},
    \end{equation}
    with
    \begin{equation}
        \bb{\delta}_{\bb{x}_A} \leqdef \left(\frac{k_i - (m - 1) x_i}{\sqrt{m - 1}}\right)_{i\in A}.
    \end{equation}
    By successively applying the triangle inequality, the identity $a^2 - b^2 = (a + b)(a - b)$ and the total variation bound from Theorem~3 in \cite{MR0478288}, we have, as $m\to \infty$,
    \begin{align}
        &\left|\sum_{\bb{k}\in \N_0^d \cap (m - 1)\mathcal{S}_d} \hspace{-3mm} P_{\bb{k},m-1}^2(\bb{x}) - \hspace{-3mm} \sum_{\bb{k}\in \N_0^d \cap (m - 1)\mathcal{S}_d} \left(\prod_{i\in \mathcal{J}} \frac{((m - 1) x_i)^{k_i}}{k_i!} e^{-(m - 1) x_i}\right)^2 \left(\frac{\phi_{[d]\backslash \mathcal{J}}(\bb{x})}{(m - 1)^{(d - |\mathcal{J}|)/2}}\right)^2\right| \notag \\
        &\quad\leq 2.1 \sum_{\bb{k}\in \N_0^d \cap (m - 1)\mathcal{S}_d} \left|P_{\bb{k},m-1}(\bb{x}) - \left(\prod_{i\in \mathcal{J}} \frac{((m - 1) x_i)^{k_i}}{k_i!} e^{-(m - 1) x_i}\right) \frac{\phi_{[d]\backslash \mathcal{J}}(\bb{x})}{(m - 1)^{(d - |\mathcal{J}|)/2}}\right| \notag \\
        &\quad= 2.1 \, \|\bb{x}_{\mathcal{J}}\|_1 \left\{\sqrt{\frac{2}{\pi e}} + \OO\big(1 \wedge (m \|\bb{x}_{\mathcal{J}}\|_1)^{-1/2}\big)\right\} + \OO\left(\frac{\|\bb{x}_{[d]\backslash \mathcal{J}}^{-1/2}\|_1}{\sqrt{m}} + \frac{\|\bb{x}_{[d]\backslash \mathcal{J}}^{-1}\|_{\infty}}{m}\right),
    \end{align}
    with the conventions $\bb{x}_A^{\scriptscriptstyle -1/2} \leqdef (x_i^{\scriptscriptstyle -1/2})_{i\in A}$, $\bb{x}_A^{\scriptscriptstyle -1} \leqdef (x_i^{\scriptscriptstyle -1})_{i\in A}$ and $\|\bb{x}_{\emptyset}^{-1/2}\|_1 = \|\bb{x}_{\emptyset}^{-1}\|_{\infty} = 0$.
    By the assumption on $\bb{x}$ in the statement of the lemma, we get
    \begin{equation}
        \begin{aligned}
            m^{(d - |\mathcal{J}|)/2} \hspace{-3mm} \sum_{\bb{k}\in \N_0^d \cap (m - 1)\mathcal{S}_d} \hspace{-3mm} P_{\bb{k},m-1}^2(\bb{x})
            &= \sum_{\bb{k}\in \N_0^d \cap (m - 1)\mathcal{S}_d} \left(\prod_{i\in \mathcal{J}} \frac{((m - 1) x_i)^{k_i}}{k_i!} e^{-(m - 1) x_i}\right)^2 \frac{\phi_{[d]\backslash \mathcal{J}}^2(\bb{x})}{(m - 1)^{(d - |\mathcal{J}|)/2}} \\
            &\quad+ \OO_{\bb{\lambda}_{\mathcal{J}}\hspace{-0.2mm},\hspace{0.3mm} \bb{x}_{[d]\backslash \mathcal{J}}}\big(m^{-1} + m^{-1/2} \ind_{\{\mathcal{J} \neq [d]\}}\big).
        \end{aligned}
    \end{equation}
    By Equation~(39) in \cite{MR4287788}, the above is
    \begin{align}
        &= \prod_{i\in \mathcal{J}} e^{-2 \lambda_i} (1 + \OO(m^{-1})) \, \left(\sum_{k=0}^{\infty} \frac{\lambda_i^{2 k}}{(k!)^2} + \OO_{\bb{\lambda}_{\mathcal{J}}}(e^{-\frac{m}{100 d^2}})\right) \big(\psi_{[d]\backslash \mathcal{J}}(\bb{x}) + \oo_{\bb{x}_{[d]\backslash \mathcal{J}}}(1)\big) \\
        &\qquad+ \OO_{\bb{\lambda}_{\mathcal{J}}\hspace{-0.2mm},\hspace{0.3mm} \bb{x}_{[d]\backslash \mathcal{J}}}\big(m^{-1} + m^{-1/2} \ind_{\{\mathcal{J} \neq [d]\}}\big) \notag \\[1mm]
        &= \psi_{[d]\backslash \mathcal{J}}(\bb{x}) \prod_{i\in \mathcal{J}} e^{-2 \lambda_i} I_0(2 \lambda_i) + \OO_{\bb{\lambda}_{\mathcal{J}}}(m^{-1}) + \oo_{\bb{\lambda}_{\mathcal{J}}\hspace{-0.2mm},\hspace{0.3mm} \bb{x}_{[d]\backslash \mathcal{J}}}(1) \, \ind_{\{\mathcal{J} \neq [d]\}},
    \end{align}
    which proves \eqref{eq:lem:Leblanc.2012.boundary.Lemma.3}.
    Equation~\eqref{eq:lem:technical.sums.power.3.eq.2} is direct corollary to Equation~(40) in \cite{MR4287788}.
    This ends the proof.
\end{proof}

The third lemma generalizes Lemma~4 in \cite{MR4287788}.
It is used in the proof of Theorem~\ref{thm:Leblanc.2012.boundary.Theorem.4}.

\begin{lemma}\label{lem:Leblanc.2012.boundary.Lemma.7}
    For any $\bb{x}\in \mathcal{S}_d$ such that $x_i = \lambda_i / m$ for all $i\in \mathcal{J}\subseteq [d]$ ($\lambda_i > 0$ is fixed $\forall i\in \mathcal{J}$) and $x_i\in (0,1)$ independent of $m$ for all $i\in [d] \backslash \mathcal{J}$, we have, as $m\to \infty$,
    \begin{equation}\label{eq:lem:Leblanc.2012.boundary.Lemma.7}
        \begin{aligned}
            &\sum_{\bb{k},\bb{\ell}\in \N_0^d \cap m\mathcal{S}_d} \hspace{-2mm} ((k_p \wedge \ell_p) / m - x_p) P_{\bb{k},m}(\bb{x}) P_{\bb{\ell},m}(\bb{x}) \\
            &\qquad=
            \begin{cases}
                \, - m^{-1} \lambda_p e^{-2 \lambda_p} (I_0(2 \lambda_p) + I_1(2\lambda_p)) + \OO_{\lambda_p}(m^{-2}), &\mbox{if } p\in \mathcal{J}, \\
                \, - m^{-1/2} \sqrt{\pi^{-1} x_p (1 - x_p)} + \oo_{x_p}(m^{-1/2}), &\mbox{if } p\in [d]\backslash \mathcal{J},
            \end{cases}
        \end{aligned}
    \end{equation}
    where $I_j, ~j\in \{0,1\},$ is defined in \eqref{eq:def:I.0.I.1}.
\end{lemma}

\begin{proof}[\bf Proof of Lemma~\ref{lem:Leblanc.2012.boundary.Lemma.7}]
    We know that the marginal distributions of the multinomial are binomial.
    Therefore, when $p\in \mathcal{J}$, the result follows from Lemma~3~(B) and Lemma~7~(B) in \cite{MR2925964}.
    When $p\in [d]\backslash \mathcal{J}$, the proof is given in Lemma~4 of \cite{MR4287788}.
\end{proof}

Next, we write the joint central moments (of order two, three and four) for the multinomial distribution.
These moments were derived, for example, by \cite{Ouimet_Stats_2021}.

\begin{lemma}[Central moments]\label{lem:Leblanc.2012.boundary.Lemma.1}
    Let $\bb{x}\in \mathcal{S}_d$.
    If $\bb{\xi} = (\xi_1,\xi_2,\dots,\xi_d)\sim \mathrm{Multinomial}\hspace{0.2mm}(m,\bb{x})$, then, for all $i,j,\ell,p\in [d]$,
    \begin{align}
        \EE\big[(\xi_i - m \, x_i)(\xi_j - m \, x_j)\big]
        &= m \, (x_i \ind_{\{i = j\}} - x_i x_j), \label{eq:thm:central.moments.eq.2} \\[1.2mm]
        %%%
        \EE\big[(\xi_i - m \, x_i)(\xi_j - m \, x_j)(\xi_{\ell} - m \, x_{\ell})\big]
        &= m \, \left(\hspace{-1mm}
            \begin{array}{l}
                2 x_i x_j x_{\ell} - \ind_{\{i = j\}} x_i x_{\ell} - \ind_{\{j = \ell\}} x_i x_j \\
                - \ind_{\{i = \ell\}} x_j x_{\ell} + \ind_{\{i = j = \ell\}} x_i
            \end{array}
            \hspace{-1mm}\right), \label{eq:thm:central.moments.eq.3} \\[0.5mm]
        %%%
        \EE\big[(\xi_i - m \, x_i)(\xi_j - m \, x_j)(\xi_{\ell} - m \, x_{\ell})(\xi_p - m \, x_p)\big]
        &= \OO(m^2), \quad \text{as } m\to \infty. \label{eq:thm:central.moments.eq.4}
    \end{align}
\end{lemma}

\iffalse
\section{Abbreviations}\label{sec:abbreviations}

\begin{tabular}{ll}
    c.d.f. & cumulative distribution function \\
    i.i.d. & independent and identically distributed \\
    MSE & mean squared error
\end{tabular}
\fi

\section*{Conflict of the interest}

The author declares no conflict of interest.

\section*{Funding information}

The author acknowledges (past) support of a postdoctoral fellowship from the NSERC (PDF) and a postdoctoral fellowship supplement from the FRQNT (B3X).
The author is currently supported by a CRM-Simons postdoctoral fellowship from the Centre de recherches math\'ematiques and the Simons Foundation.

\section*{Acknowledgments}

We thank the anonymous referees for their comments.

%
% ----------  B I B L I O G R A P H Y  ----------
%

\bibliographystyle{apalike}
\bibliography{Ouimet_2022_Bernstein_boundary_bib}

\end{document}